\documentclass{amsart}

\usepackage{amsfonts}
\usepackage{amssymb}
\usepackage{enumerate}
\usepackage{amsmath}
\usepackage{amsthm}
\usepackage{graphicx}

\numberwithin{equation}{section}

\newtheorem{theorem}{Theorem}[section]
\newtheorem{lemma}[theorem]{Lemma}

\theoremstyle{definition}
\newtheorem{definition}[theorem]{Definition}
\newtheorem{notation}[theorem]{Notation}
\newtheorem{remark}[theorem]{Remark}

\author{Rich\'ard Balka}
\address{Institute of
Mathematics and Informatics, Eszterh\'azy K\'aroly College,
Le\'anyka str.~4., H-3300 Eger, Hungary}
\email{balkar@cs.elte.hu}
\thanks{Partially supported by Hungarian Scientific Foundation grant no. 72655.}

\date{}

\title[Duality between measure and category in LCA Polish groups]{Duality between measure and category in uncountable locally compact abelian Polish groups}

\begin{document}

\begin{abstract} We show that there is no addition preserving Erd\H{o}s-Sierpi\'nski mapping
 on any uncountable locally compact abelian Polish group.
 This generalizes results of Bartoszy\'nski and Kysiak.
\end{abstract}

\subjclass[2010]{Primary: 03E15, 22B05; Secondary: 28A99}

\keywords{Measure, category, duality, Erd\H{o}s-Sierpi\'nski mapping, LCA groups}

\maketitle

\section{Introduction.}

Let $G$ be a locally compact abelian (LCA) Polish group. Let $\mathcal{M}$
and $\mathcal{N}$ be the ideals of meager and null (with respect to Haar
measure) subsets of $G$.

\begin{definition} A bijection $F\colon G\rightarrow G$ is called an
\emph{Erd\H{o}s-Sierpi\'nski mapping} if $$ X\in \mathcal{N} \Leftrightarrow
F[X]\in \mathcal{M} \quad \textrm{and} \quad X\in \mathcal{M} \Leftrightarrow
F[X]\in \mathcal{N}.
$$ \end{definition}

\begin{theorem} \textbf(Erd\H{o}s-Sierpi\'nski) Assume the Continuum
Hypothesis. Then there exists an Erd\H{o}s-Sierpi\'nski mapping on
$\mathbb{R}$. \end{theorem}

The existence of such a function is independent from ZFC. Our main
question is the following:

Is it consistent that there is an Erd\H{o}s-Sierpi\'nski mapping
$F$ that preserves addition, namely $$\forall x,y \in G \quad
F(x+y)=F(x)+F(y)?$$

This question is attributed to Ryll-Nardzewski in the case $G=\mathbb{R}$.
Besides intrinsic interest, another motivation was the following: If
this statement were consistent then the so called strong measure
zero and strongly meager sets would consistently form isomorphic
ideals. (For the definitions see \cite{Car}.)

First Bartoszy\'nski gave a negative answer to the question in the
case $G=2^{\omega}$, see \cite{Bar}, then Kysiak proved this for
$G=\mathbb{R}$ and answered the question of Ryll-Nardzewski, see
\cite{Kys}, where he used and improved Bartoszy\'nski's idea. We
answer the general case, the goal of this article is to prove the
following theorem:

\begin{theorem} \textbf(Main Theorem) There is no addition preserving
Erd\H{o}s-\linebreak Sierpi\'nski mapping on any uncountable locally compact
abelian Polish group. \end{theorem}

Let ($\varphi _{\mathcal{M}}$) denote the following statement (considered
by Carlson in \cite{Car}): For every $F\in \mathcal{M}$ there exists a set
$F'\in \mathcal{M}$ such that
$$\forall x_1,x_2\in G \, \exists x\in G \quad
(F+x_1)\cup(F+x_2)\subseteq F'+x. $$

Let ($\varphi _{\mathcal{N}}$) be the dual statement obtained by replacing
$\mathcal{M}$ by $\mathcal{N}$. If there exists an Erd\H{o}s-Sierpi\'nski mapping
preserving addition then ($\varphi _{\mathcal{M}}$) and ($\varphi _{\mathcal{N}}$)
are equivalent.  First we show that ($\varphi _{\mathcal{M}}$) holds in
LCA Polish groups. In the second part we begin to show that
($\varphi _{\mathcal{N}}$) fails for all uncountable LCA Polish groups by
reducing the general case to three special cases. Finally, in part
three we settle these three special cases.

\section{($\varphi _{\mathcal{M}}$) Holds For All LCA Polish Groups.}

As the known proofs only work for the reals, we had to come up with
a new, topological proof.

\begin{notation} Let $X$ be a metric space, $x\in X$ and $r>0$. Let $B(x,r)$
denote the closed ball of radius $r$ centered at the point $x$.
\end{notation}

\begin{lemma} Let $X$ be a metric space and $K\subseteq X$ a nowhere dense
compact set. Then there exists a function $f\colon \mathbb{R}^{+}
\rightarrow \mathbb{R}^{+}$ such that for all $x\in X$ and $r>0$ there
exists $y\in X$ such that $B(y,f(r)) \subseteq B(x,r) \setminus K$.
\end{lemma}

\begin{proof} Suppose towards a contradiction that $r>0$ is such a number
that there exists a sequence $r _{n} \to 0 $ $ (n\in \mathbb{N})$ and a set
$\{x_{n}\in X:\,n\in \mathbb{N}\}$ such that for all $y\in X$ and $n\in
\mathbb{N}$
\begin{equation} \label{p1} B(y,r _{n}) \nsubseteq B(x_{n},r)\setminus K.
\end{equation}
$r_{n}<r$ holds for large enough $n$, so in the case $y=x_{n}$ we
get that for large enough $n$ there exist $z_{n}\in B(x_{n}, r _{n})
\cap K$. By the compactness of $K$ there exists a convergent
subsequence $\lim _{k \to \infty} z_{n_k}=z\in K$, and so $\lim _{k
\to \infty} x_{n_k}=z\in K$. There is an $N\in \mathbb{N}$ such that
$x_{n_k}\in B(z,\frac{r}{2})$ holds for all $k>N$. Then
$B\left(z,\frac{r}{2}\right)\subseteq B(x_{n_k},r)$, so by
\eqref{p1} for all $k>N$ and $y\in B\left(z,\frac{r}{2}\right)$ we
get $B(y,r _{n_k}) \nsubseteq B\left(z,\frac{r}{2}\right)\setminus
K$, which contradicts that $K$ is nowhere dense, and we are done.
\end{proof}

Every abelian Polish group admits a compatible invariant complete
metric, because it admits a compatible invariant metric by
[\cite{GH}, Thm. 7.3.], and a compatible invariant metric is
automatically complete by [\cite{GH}, Lem. 7.4.]. So we may assume
that the metric on our group is invariant.

\begin{lemma} \label{pl} Let $G$ be an abelian Polish group  and $K\subseteq G$
a nowhere dense compact set. Then there exists a function $l\colon
\mathbb{R}^{+} \rightarrow \mathbb{R}^{+}$  such that for all $x,x_{1},x_{2}\in G$
and $r>0$ there is a $y\in G$ such that
$$B(y,l(r)) \subseteq B(x,r) \setminus
\left((K+x_{1}) \cup (K+x_{2})\right).$$ \end{lemma}

\begin{proof}  It is easy to see, using the invariance of our metric, that
$l=f\circ f$ works, where $f$ is defined in the previous lemma.
\end{proof}

We may assume that $l(r)<r$ for all $r>0$. Now we prove ($\varphi
_{\mathcal{M}}$).

\begin{theorem} \label{thmst} Let $S$ be a meagre set in an
LCA Polish group $G$. Then there is a meagre set $T\subseteq G$
such that for all $s_{1},s_{2}\in G$ there is a $t\in G$ such that
$(S+s_{1})\cup (S+s_{2}) \subseteq T+t$. \end{theorem}

\begin{proof}   We may assume by local compactness (by taking closures of the
nowhere dense subsets and decomposing each of them to countably many
compact sets) that $S=\cup _{n\in \mathbb{N}} S_{n}$, where the $S_n
\textrm{'s } (n\in \mathbb{N})$ are nowhere dense compact sets. The idea is
that we construct nowhere dense $T_{n}\textrm{'s}$ for the
$S_{n}\textrm{'s}$ simultaneously, and $T$ will be the union of the
$T_{n}\textrm{'s}$. Later we will set $T_{n}=\cap _{k\in \mathbb{N}}
T_{n}^{k}$ for some open sets $T_{n}^{k}$, and we will
simultaneously construct a decreasing sequence of closed balls
$B(x_{k},r_{k})$ such that $t=t_{s_{1},s_{2}}$ will be found as
$\cap _{k\in \mathbb{N}} B(x_{k},r_{k})$.

Let $h\colon \mathbb{N} \rightarrow \mathbb{N}\times \mathbb{N}$ be an enumeration of
$\mathbb{N} \times \mathbb{N}$, let $h_{1},h_{2}\colon \mathbb{N}\rightarrow \mathbb{N}$ be the
first and second coordinate functions of $h$ and let $\{ g_n: \,
n\in \mathbb{N} \}$ be a dense set in $G$. Let $l_{n}$ be the function in
the previous lemma for the compact set $S_n$. We define $T$ by
induction, independently from $s_1,s_2$. Let $x_{0}\in G$ be
arbitrary and $r_{0}>0$ so small that $B(x_{0},r_{0})$ is compact
(we can do this by local compactness of $G$), and
$r_k=\frac{l_{h_{1} (k)}(r_{k-1})}{2}$ if $k>0$ and $n\in \mathbb{N}$. Let

\begin{equation} \label{0def} T_{n}^{k} =
\begin{cases}
G & \textrm{ if } n\neq h_{1}(k),\\
G\setminus B(g_{h_{2}(k)},r_k) & \textrm{ if } n=h_{1}(k).
\end{cases}
\end{equation}

The key step of the proof is the following lemma.

\begin{lemma} Assume that $x_{0},x_{1},\dots, x_{k-1}$ are already defined. Then
there exists $x_k \in G$ (which depends on $s_{1}$ and $s_{2}$) such
that $B(x_k,r_k)\subseteq B(x_{k-1},r_{k-1})$ and for all $n\in \mathbb{N}$
and $t\in B(x_{k},r_k)$
\begin{equation} \label{eq} (S_{n}+s_{1}) \cup (S_{n}+s_{2})\subseteq T_{n}^{k}+t
.\end{equation} \end{lemma}
\begin{proof}  (of the Lemma) If $n=h_{1}(k)$, then $T_{n}^{k}=G\setminus
B(g_{h_{2}(k)},r_k)$. Using
$2r_{k}=l_{h_{1}(k)}(r_{k-1})=l_{n}(r_{k-1})$ and Lemma \ref{pl}
there exists $y_{k} \in B(x_{k-1}+g_{h_{2}(k)},r_{k-1})$ such that
\begin{equation} \label{p2} B(y_{k},2r_k)\subseteq
B(x_{k-1}+g_{h_{2}(k)},r_{k-1})\setminus \left((S_{n}+s_{1}) \cup
(S_{n}+s_{2})\right). \end{equation}
Let us define $x_k:= y_{k}-g_{h_{2}(k)}$. Using the definition of
$x_k$ and \eqref{p2} we get that
$$B(x_k,r_k)=B(y_k,r_k)-g_{h_{2}(k)}\subseteq
B(y_k,2r_k)-g_{h_{2}(k)}\subseteq B(x_{k-1},r_{k-1}).$$ We will use
the following easy equation, where the first $+$ is the Minkowski
sum.
\begin{equation} \label{Mink} B(g_{h_{2}(k)},r_k)+B(x_k,r_k)\subseteq
B(g_{h_{2}(k)}+x_{k},2r_{k}). \end{equation}
Using \eqref{p2} again, the definition of $x_{k}$ and \eqref{Mink}
in this order we get that for all $t\in B(x_k,r_k)$
\begin{align*}(S_{n}+s_{1}) \cup (S_{n}+s_{2})&\subseteq G\setminus
B(y_k,2r_k) \\  
&=G\setminus B(g_{h_{2}(k)}+x_{k},2r_k) \\ 
&\subseteq G \setminus B(g_{h_{2}(k)},r_{k})+t.
\end{align*}
Hence $$(S_{n}+s_{1}) \cup (S_{n}+s_{2})\subseteq T_{n}^{k}+t,$$
so \eqref{eq} holds for $n=h_{1}(k)$. If $n\neq h_{1}(k)$ then
$T_{n}^{k}=G$ and \eqref{eq} is obvious, so we are done.
\end{proof}

Now we return to the proof of Theorem \ref{thmst}.
By the compactness of $B(x_{0},r_{0})$ the closed sets $B(x_k,r_k)$
are compact, so the intersection of decreasing sequence of compact
sets $\cap _{k\in \mathbb{N}} B(x_k,r_k)\neq \emptyset$. Let
$$t_{s_{1},s_{2}}\in \cap _{k\in \mathbb{N}} B(x_k,r_k)$$ be
the common shift. By (\ref{eq}) and the definition of
$t_{s_{1},s_{2}}$,

\begin{equation} \label{0eq} (S_{n}+s_{1}) \cup (S_{n}+s_{2})\subseteq
T_{n}^{k}+t_{s_{1},s_{2}} \end{equation}
holds for all $k,n\in \mathbb{N}$. For every $n\in \mathbb{N}$ the set
$T_{n}=\cap _{k\in \mathbb{N}} T_{n}^{k}$ is nowhere dense by \eqref{0def}.
By \eqref{0eq} we easily get for every $n\in \mathbb{N}$
$$(S_n+s_1) \cup (S_n+s_2)\subseteq T_{n}+t_{s_{1},s_{2}}.$$
The set $T=\cup _{n\in \mathbb{N}} T_{n}$ is meagre, and clearly
$$(S+s_1) \cup (S+s_2)\subseteq T+t_{s_{1},s_{2}},$$
and the proof is complete.
\end{proof}

\section{Reduction to $\mathbb{T}$, $\mathbb{Z} _{p}$ and $\prod _{n\in \mathbb{N}} G_{n}$.}

In this section we reduce the general uncountable LCA Polish
groups to some special groups. We follow the strategy developed in
\cite{MEAT}.

\begin{definition} Let us say that an LCA Polish group $G$ is \emph{nice} if
($\varphi _{\mathcal{N}}$) fails in $G$, that is, there is a nullset $N$,
such that for every nullset $N'$ there are $x_1,x_2\in G$ such that
$(N+x_1) \cup (N+x_2) \nsubseteq (N'+x)$ for all $x\in G$.
\end{definition}

\begin{lemma} \label{l1} If an LCA Polish group $G$ has a nice open subgroup $U$
then $G$ is nice.  \end{lemma}

\begin{proof}  Let $\mu$ be the Haar measure of the LCA Polish subgroup $U$ and
$N\subseteq U$ be the set that witnesses that $U$ is nice. Because
of $U$ is open and $G$ is separable, we can write a disjoint
countable decomposition $G=\cup _{n=0} ^{\infty} (U+g_{n})$. It is
easy to see that $\nu(B)=\sum _{n=0} ^{\infty} \mu ((B-g_{n})\cap
U)$ is a Haar measure on $G$. We show that the $\nu$ nullset
$N_{G}=\cup _{n=0} ^{\infty} (N+g_{n})$ is a witness that $G$ is
nice. Let $\nu(N'_{G})=0$. Since $N$ witnesses that $U$ is nice,
there are $u_1,u_2 \in U$ such that

\begin{equation}  \left((N+u_{1}) \cup (N+u_{2})\right) \cap \left((N'_{G}\cap
U)+u\right)^c\neq \emptyset \textrm{ for all } u\in U.
\end{equation}

Suppose towards a contradiction that there exists a $g\in G$ such
that \\ $(N_{G}+u_1) \cup (N_{G}+u_2)\subseteq (N'_{G}+g)$. There
exists $i\in \mathbb{N}$ such that $g-g_{i}\in U$. Then
$$(N+u_1) \cup (N+u_2) \subseteq \left(N'_{G}+(g-g_{i})\right)\cap
U=(N'_{G}\cap U)+(g-g_i)
$$ where $g-g_{i}\in U$, a contradiction.
\end{proof}

\begin{lemma} \label{l2}  Assume that $G$ is an LCA Polish group, $H\subseteq G$
is a compact subgroup and $G/H$ is nice. Then $G$ is nice, too.
\end{lemma}

\begin{proof}   If $\mu$ is a Haar measure on $G$, $\pi \colon G\rightarrow G/H$ is the
canonical homomorphism, then $\mu _{G}\circ \pi^{-1}$ is a Haar
measure on $G/H$  by [\cite{Ha}, 63. Thm. C.]. So the inverse image
of a nullset in $G/H$ under $\pi$ is a nullset in $G$. Hence if
$N\subseteq G/H$ is a nullset witnessing that $G/H$ is nice then
$\pi ^{-1}(N)\subseteq G$ is a nullset witnessing that $G$ is also
nice since the translated copies of $\pi ^{-1}(N)$ are composed of
cosets of $H$, so if we would like to cover it with a translation of
a set $N'$, then we may assume without loss of generality, that $N'$
consists of cosets, too. It is easy to see that $\pi(N')$ shows that
$N$ is not a witness, a contradiction.
\end{proof}

Now we start reducing the problem to simpler groups.

\begin{definition} Let $G$ be a group and $p$ be a prime number, $G_{p^n}=\{g\in G:
p^{n}g=0 \}$ for every $n\in \mathbb{N}$, and also let $G_{p^{\infty}}=\cup
_{n\in \mathbb{N}} G_{p^n}$. We say that $G$ is a \emph{p-group}, if
$G=G_{p^{\infty}}$. \end{definition}

\begin{definition} Let $p$ be a prime number. An abelian group $G$ is called
\emph{quasicyclic} if it is generated by a sequence $\langle
g_{n}\rangle_{n\in \mathbb{N}}$ with the property that $g_{0}\neq 0$ and
$pg_{n+1}=g_{n}$ for every $n\in \mathbb{N}$. For a fixed prime $p$ the
unique (up to isomorphism) quasicyclic group is denoted by
$C_{p^{\infty}}$.
\end{definition}

\begin{notation} We denote by $\mathbb{T}$ the circle group, by $\mathbb{Z}_{p}$ the group of
$p$-adic integers for every prime $p$ and by $\prod _{n\in \mathbb{N}}
G_{n}$ the product of the finite abelian groups $G_n$'s.
\end{notation}

\begin{remark} \label{rem3} $\mathbb{Z}_{p}$ is the topological space
$\{0,1,\dots,p-1\}^{\omega}$ equipped with the product topology.
Addition is coordinatewise with carried digits from the $n^{th}$
coordinate to the $n+1^{st}$. Both $\mathbb{Z} _p$ and $\prod _{n\in \mathbb{N}}
G_{n}$ are Polish with the product topology. \end{remark}

We will use the following theorem from \cite{MEAT}.

\begin{theorem} Every infinite abelian group contains a subgroup isomorphic to one
of the following:
\begin{itemize}
\item $\mathbb{Z},$
\item $C_{p^{\infty}}$ for some prime $p$,
\item $\bigoplus _{n\in \mathbb{N}} G_{n}$, where each $G_n$ is a finite
abelian group of at least two elements.
\end{itemize}
\end{theorem}

\begin{theorem} $\mathbb{T}$, $\mathbb{Z} _{p}$ and $\prod _{n\in \mathbb{N}} G_{n}$ are nice. \end{theorem}

\begin{proof}  We postpone the proof to the next section.
\end{proof}

\begin{theorem} Every uncountable LCA Polish group is nice, that is, ($\varphi
_{\mathcal{N}}$) fails for every uncountable LCA Polish group. \end{theorem}

\begin{proof}  Let $G$ be an LCA Polish group. From the Principal Structure
Theorem of LCA groups [\cite{Ru}, 2.4.1] follows that $G$ has an
open subgroup $H$ which is of the form $H=K\otimes \mathbb{R}^{n}$, where
$K$ is a compact subgroup and $n\in \mathbb{N}$. $G$ is uncountable, so
nondiscrete and $H$ is open, therefore $H$ is a nondiscrete, so
uncountable group. By \ref{l1} it is enough to prove that $H$ is
nice, so we can assume $G=H$.

Assume that $n\geq 1$. $\mathbb{R}$ is nice by \cite{Kys}, let $N$ be a
nullset witnessing this fact. It is obvious that $K\times N \times
\mathbb{R}^{n-1}$ witnesses that $G=K\otimes \mathbb{R}^{n}$ is nice. So we may
assume $n=0$, and we get that $G$ is compact. It is enough to find a
closed subgroup $H\subseteq G$ such that $G/H$ is nice by the
previous lemma. Using [\cite{Ru}, 2.1.2] and the Pontryagin Duality
Theorem, see [\cite{Ru}, 1.7.2]), we get that factors of $G$ are the
same as (isomorphically homeomorphic to) dual groups of closed
subgroups of $\widehat{G}$. If $G$ is compact, then $\widehat{G}$ is
discrete, see [\cite{Ru}, 1.2.5]. So it is enough to find a subgroup
$M\subseteq \widehat{G}$ such that $\widehat{M}$ is nice.

From the previous theorem follows (using that $G=K$ is an infinite
group) that $\widehat{G}$ has a subgroup isomorphic either to $\mathbb{Z}$,
or to $C_{p^{\infty}}$ for some prime $p$, or to $\oplus _{n\in \mathbb{N}}
G_{n}$ (where each $G_n$ is a finite abelian group of at least two
elements). It is sufficient to show that the duals of these groups
are nice. It is well-known that $\widehat{\mathbb{Z}}=\mathbb{T}$, and $\mathbb{T}$ is
nice, see \cite{Kys}.

By [\cite{Ru}, 2.2.3] the dual of a direct sum (equipped with the
discrete topology) is the direct product of the dual groups
(equipped with the product topology), so $\widehat{(\oplus _{n\in
\mathbb{N}} G_{n})}=\prod _{n\in \mathbb{N}} \widehat{G_{n}}$. If $G_{n}$ is
finite so is $\widehat{G_{n}}$, so we are done with the second case.

Finally, $\widehat{C_{p^{\infty}}}=\mathbb{Z}_{p}$, see [\cite{HR}, 25.2]. Hence
we are done by the previous theorem.
\end{proof}

\section{($\varphi _{\mathcal{N}}$) Fails For $\mathbb{T}$, $\mathbb{Z} _{p}$
And $\prod _{n\in \mathbb{N}} G_{n}$.}

We will prove that $\mathbb{T}$, $\mathbb{Z} _{p}$  and $\prod _{n\in \mathbb{N}} G_{n}$
are nice. $\mathbb{T}$ is nice, see \cite{Kys}, so we need to handle the
last two cases. Our proofs are very similar to the Main Theorem of
Kysiak's paper, see [\cite{Kys}, Main Thm. 4.3.], so we only
describe the necessary modifications. Hence in order to follow our
proof one has to read Kysiak's paper parallelly. Unless stated
otherwise, the numbered references in the following proofs refer to
the present paper, and the page numbers refer to \cite{Kys}.

\begin{theorem} \label{t1} $\mathbb{Z}_{p}$ is nice. \end{theorem}

\begin{proof}  We use the description of $\mathbb{Z}_{p}$ that can be found in Remark
\ref{rem3}. We always write $\mathbb{Z} _{p}$ instead of $(0,1]$. Let $\mu$
be the Haar measure on $\mathbb{Z}_{p}$, and let $\mu_{2}=\mu\times \mu$ be
the Haar measure on $\mathbb{Z}_{p}^{2}$. The proof of Lemma 4.7. in
\cite{Kys} is the same with these modifications.

For the construction of $I_{n}\textrm{'s}$ we fix a partition of the
set $\omega$ instead of $\omega \setminus\{0\}$. For the
inequalities we write $p^{I_n}$ and $p^{|I_n|}$ instead of $2^{I_n}$
and $2^{|I_n|}$. With these changes the following inequality holds.
\begin{equation} \label{pol1} \frac{1}{p^{|I_n|}} \leq
\frac{1}{12}\left(\frac{1}{n^{2}}-\frac{1}{n^5}\right) \textrm{ for
} n>1.  \end{equation}
Let $J_{n}\subseteq p^{I_{n}}$ such that
\begin{equation} \label{pol2}
1-\frac{1}{n^{2}}+\frac{1}{n^{5}}>\frac{|J_n|}{p^{|I_n|}}>1-\frac{1}{n^{2}},
\end{equation}
and $J_n$ consist of the first $|J_n|$ consecutive elements of
$p^{I_n}$ with respect to the antilexicographical ordering
(sequences are ordered according to the largeness of the rightmost
coordinate where they differ).

Let $\mathcal{N}^{*}$ be the filter of full measure sets of $\mathbb{Z}_{p}$. If
$I\subseteq \omega$ and $J\subseteq p^{I}$ then $[J]$ denote the set
$\{ x\in \mathbb{Z}_{p}: \, x\upharpoonright I\in J\}$. We write $
p^{<\omega}$ instead of $2^{<\omega}$ and for $s\in p^{<\omega}$ let
$[s]$ be the set $\{ x\in \mathbb{Z}_{p}: x \textrm{ extends } s \}$.

For the definition of $C$ (at the end of the page 274) we omit the
$C\subseteq \mathbb{R}\setminus \mathbb{Q}$ condition. We may assume about $C$ the
following: for every $s\in p^{<\omega}$ we have $[s]\cap
C=\emptyset$ or $\mu([s]\cap C)>0$
 (if not, consider $C'=C\setminus \cup \{[s]: s\in p^{<\omega}
\textrm{ and } \mu([s]\cap C)=0\}$ instead).

The inductive definition and its proof (beginning on page 275) is
almost the same, too. Write $p^{I_{1}\cup \dots \cup I_{l}}$ and
$p^{[\max I_{l},\infty)}$ instead of $2^{I_{1}\cup \dots \cup
I_{l}}$ and $2^{[\max I_{l},\infty)}$. Let $0^{n}$ be the element of
$p^{\{0,\dots,n-1\}}$ with zero coordinates. For the proof of
\begin{equation} \label{mu} \mu([r_{s}]\cap C)>\frac{t-1}{t}\mu([r_{s}]) \end{equation}
(that can be found at the end of the page 275) we use (instead of
the Lebesgue density theorem) the following more general density
theorem.

\begin{definition} \label{def} Let $X$ be a topological group.
A \emph{B-sequence} in $X$ is
a non-increasing sequence $\langle V_{n}\rangle_{n\in \mathbb{N}}$ of
closed neighborhoods of the identity, constituting a base of
neighborhoods of the identity, such that there is some $M\in \mathbb{N}$
such that for every $n\in \mathbb{N}$ the set $V_{n}-V_{n}$ can be covered
by at most $M$ translates of $V_{n}$. \end{definition}

\begin{theorem} \label{thm} Let $X$ be a topological group with a left Haar
measure $\mu$, and $\langle V_{n}\rangle_{n\in \mathbb{N}}$ a B-sequence in
$X$. Then for any Haar measurable set $E\subseteq X$,
$$ \lim_{n\to \infty} \frac{\mu \left(E\cap(x+V_{n})\right)}{
\mu(V_n)}=\chi _{E} (x)$$ for almost every $x\in X$. \end{theorem}

For the proof of Theorem \ref{thm} and for Definition \ref{def} see
[\cite{Fre}, 447D Thm.] and [\cite{Fre}, 446L Def.]. It is easy to
see that $V_{n}=[0^{n+1}]$ $(n\in \mathbb{N})$ form a
B-sequence, and \eqref{mu} follows by applying
Theorem \ref{thm} for $\langle V_{ n} \rangle_{n\in \mathbb{N}}$.

\begin{remark} \label{rem4} For every countable product space of discrete
groups $\prod _{n\in \mathbb{N}} X_{n}$ the sets
$V_{n}=\left[\{e_{0},\dots,e_{n}\}\right]$ $(n\in \mathbb{N})$ form a B-sequence
(where $e_{i}$ is the neutral element of $X_{i}$), so \eqref{mu} is
valid, too. \end{remark}

We follow the proof till the definition of $U$ (at the end of the
page 277), then we deviate from it. We jump to the formula of the
following lemma that is at the end of the proof in \cite{Kys} (in
the middle of the page 279). For the following lemma we need a new
idea, Kysiak's arguments are not applicable here, he uses a specific
relation between the metric and the addition in $\mathbb{T}$.

\begin{lemma} \label{plem} For every $k\in U$ there is a $v_k \in [J'_{n_k}]$
such that
$$z\upharpoonright
I_{n_k}=v_{k}\upharpoonright I_{n_k} \Rightarrow z\notin
(x-y)+[J_{n_k}].$$ \end{lemma}

\begin{proof}  (of the Lemma) Assume $k\in U$. Remember that $\varepsilon
_{n_k}=\frac{1}{4}\lambda _{n_k}$, and as a consequence of
\eqref{pol2} $\lambda _{n_k}=1-\mu(J_{n_k})\geq
\frac{1}{n_{k}^{2}}-\frac{1}{n_{k}^{5}}$. Using the definition of
$U$, the previous facts and \eqref{pol1} in this order we get
\begin{align*} \mu\left([J'_{n_k}]\setminus
\left((x-y)+[J_{n_k}]\right)\right)&=\mu\left(\left(y+[J'_{n_k}]\right)\setminus
\left(x+[J_{n_k}]\right)\right) \\ 
&\geq \frac{\lambda
_{n_k}-\varepsilon _{n_k}}{2}>\frac{\lambda _{n_k}}{12} \\ 
&\geq \frac{1}{12}\left(\frac{1}{n_{k}^{2}}-\frac{1}{n_{k}^{5}}\right)\geq
\frac{1}{p^{|I_{n_k}|}}.
\end{align*}
Hence \begin{equation} \label{pol3} \mu\left([J'_{n_k}]\setminus
\left((x-y)+[J_{n_k}]\right)\right)>\frac{1}{p^{|I_{n_k}|}}.
\end{equation}
Let us define
\begin{equation} \label{pol3,5}
L_{k}:=\{\left((x-y)+[J_{n_k}]\right)\upharpoonright
I_{n_k}\}\subseteq p^{I_{n_k}}. \end{equation}
 Since the elements of $J_{n_k}$
are consecutive (with respect to the antilexicographical ordering),
we get $|L_{k}|\leq |J_{n_k}|+1$. Using this inequality and the
translation invariance of $\mu$
\begin{equation} \label{pol4}
\mu\left((x-y)+[J_{n_k}]\right)=\mu\left([J_{n_k}]\right)\geq
\mu\left([L_{k}]\right)-\frac{1}{p^{|I_{n_k}|}}. \end{equation}
 Using $(x-y)+[J_{n_k}]\subseteq [L_{k}]$, \eqref{pol4} and
 finally \eqref{pol3} we get
$$\mu\left([J'_{n_k}]\setminus [L_{k}]\right)\geq
\mu\left([J'_{n_k}]\setminus
\left((x-y)+[J_{n_k}]\right)\right)-\frac{1}{p^{|I_{n_k}|}}>0.$$
Hence there is a $v_{k}\in [J'_{n_k}]\setminus [L_{k}]\neq
\emptyset$, and it means by the definition of $L_{k}$ that $$
z\upharpoonright I_{n_k}=v_{k}\upharpoonright I_{n_k} \Rightarrow
z\notin (x-y)+[J_{n_k}],$$ and Lemma \ref{plem} follows.
\end{proof}

We can simply follow the last few lines of Kysiak's proof after the
analogous formula.
\end{proof}

\begin{theorem} \label{t2} $\prod _{n\in \mathbb{N}} G_{n}$ is nice. \end{theorem}

\begin{proof}  We write $\mathbb{N}$ and $\prod _{n\in \mathbb{N}} G_{n}$ instead of $\omega$
and $(0,1]$. Let $\mu$ be the Haar measure on $\prod _{n\in \mathbb{N}}
G_{n}$, and let $\mu_{2}=\mu\times \mu$ be the Haar measure on
$\left(\prod _{n\in \mathbb{N}} G_{n}\right)^{2}$. The proof of Lemma 4.7.
in \cite{Kys} is the same.

For the construction of $I_{n}\textrm{'s}$ we fix a partition of the
set $\mathbb{N}$ instead of $\omega \setminus\{0\}$. Write $\prod _{i\in
I_n} G_i$ and $|\prod _{i\in I_n} G_i|$ instead of $2^{I_n}$ and
$2^{|I_n|}$. At the choice of the sets $J_{n}$ we omit the ordering
condition.

Let $\mathcal{N}^{*}$ be the filter of full measure sets of $\prod _{n\in
\mathbb{N}} G_{n}$. If $I\subseteq \mathbb{N}$ and $J\subseteq \prod _{n\in I} G_n
$ then let $[J]=\left\{x\in \prod _{n\in \mathbb{N}} G_n: \,
x\upharpoonright I\in J\right\}$. We write $\left\{\prod _{i\in I}
G_i: |I|\nolinebreak <\infty\right\}$ instead of $2^{<\omega}$ and
for $s\in \left\{\prod _{i\in I} G_i: |I|\nolinebreak
<\infty\right\}$ let $[s]=\left\{ x\in \prod _{n\in
\mathbb{N}} G_{n}: x \textrm{ extends } s \right\}$.

For the definition of $C$ we omit the $C\subseteq \mathbb{R}\setminus \mathbb{Q}$
condition, and we may assume that for every $s\in  \left\{\prod
_{i\in I} G_i: |I|\nolinebreak <\infty\right\}$ we have $[s]\cap
C=\emptyset$ or $\mu([s]\cap C)>0$. \eqref{mu} follows from Remark
\ref{rem4} instead of the Lebesgue density theorem.

The inductive definition and its proof need the following
modifications. Use $\prod _{i\in I_{1}\cup \dots \cup I_{l}} G_i$
and $\prod _{i=\max I_{l}} ^{\infty} G_i$ instead of $2^{I_{1}\cup
\dots \cup I_{l}}$ and $2^{[\max I_{l},\infty)}$.

We follow the proof till the definition of $U$, then we deviate from
it. Note that by the definition of $U$, for every $k\in U$
\begin{equation} \label{pol5} \mu\left([J'_{n_k}]\setminus
\left((x-y)+[J_{n_k}]\right)\right)=\mu\left(y+[J'_{n_k}]\right)\setminus
\left(\left(x+[J_{n_k}]\right)\right)>0. \end{equation}

We jump to the formula of the
following lemma that is at the end of the proof in \cite{Kys}.
The following lemma is analogous to Lemma \ref{plem},
but the proof is much easier than in the case $\mathbb{Z}_{p}$.

\begin{lemma} \label{plem2} For every $k\in U$ there is a $v_k \in [J'_{n_k}]$
such that
$$   z\upharpoonright
I_{n_k}=v_{k}\upharpoonright I_{n_k} \Rightarrow z\notin
(x-y)+[J_{n_k}].$$ \end{lemma}

\begin{proof} (of the Lemma) Let us define
$$L_{k}:=\{\left((x-y)+[J_{n_k}]\right)\upharpoonright
I_{n_k}\}\subseteq \prod _{i\in I_{n_k}} G_{i},$$
then obviously $(x-y)+[J_{n_k}]=[L_{k}]$. Using \eqref{pol5} we get
$$\mu\left([J'_{n_k}]\setminus [L_{k}]\right)=
\mu\left([J'_{n_k}]\setminus
\left((x-y)+[J_{n_k}]\right)\right)>0.$$
So there is a $v_{k}\in [J'_{n_k}]\setminus [L_{k}]\neq \emptyset$,
and it means by the definition of $L_{k}$ that
$$  z\upharpoonright I_{n_k}=v_{k}\upharpoonright I_{n_k}
\Rightarrow z\notin (x-y)+[J_{n_k}],$$ and Lemma \ref{plem2} follows.
\end{proof}

The last few lines of the proof is the same as in Kysiak's paper
after the analogous formula.
\end{proof}

\begin{remark} Theorem \ref{t2} could be proved easier based on
Bartoszy\'nski's paper \cite{Bar}, but we do not know to generalize
Bartoszy\'nski's method to the case $\mathbb{Z}_{p}$ in Theorem \ref{t1}.
\end{remark}

\noindent \textbf{Acknowledgement} The author is indebted to Tomek
Bartoszy\'nski, M\'arton Elekes and Slawomir Solecki for some ideas.


\begin{thebibliography}{99}

\bibitem{Car} T. J. Carlson, \textit{Strong measure zero and
strongly meager sets}, Proc. Amer. Math. Soc., \textbf{118} (1993),
no. 2, 577-586.

\bibitem{Bar} T. Bartoszy\'nski, \textit{A note on duality between
measure and category}, Proc. Amer. Math. Soc., \textbf{128} (2000),
no. 9, 2745--2748.

\bibitem{Kys} M. Kysiak, \textit{Another note on duality between
measure and category}, Bull. Pol. Acad. Sci. Math., \textbf{51}
(2003), no. 3, 269--281.

\bibitem{Ha} P. R. Halmos, \textit{Measure theory}, Springer,
1974.

\bibitem{MEAT} M. Elekes and \'A. T\'oth,
\textit{Covering locally compact groups by less than $2^\omega$ many
translates of a compact nullset}, Fund. Math., \textbf{197} (2007),
243--257.

\bibitem{Ru} W. Rudin, \textit{Fourier analysis on groups}, John
Wiley \& Sons, 1990.

\bibitem{HR} E. Hewitt and K. A. Ross, \textit{Abstract harmonic
analysis}, Vol. I., Springer, 1963.

\bibitem{GH} G. Hjorth, \textit{Classification and orbit equivalence
relations}, American Mathematical Society, Providence, RI, 2000.

\bibitem{Fre} D. H. Fremlin, \textit{Measure Theory, vol. 4,
Topological Measure Spaces}, Torres Fremlin, 2003.

\end{thebibliography}
\end{document}